\documentclass{amsart}
\usepackage{latexsym,amsxtra,amscd,ifthen}
\usepackage{amsfonts}
\usepackage{verbatim}
\usepackage{amsmath}
\usepackage{amsthm}
\usepackage{amssymb}
\numberwithin{equation}{section}

\theoremstyle{plain}
\newtheorem{theorem}{Theorem}[section]
\newtheorem{lemma}[theorem]{Lemma}
\newtheorem{proposition}[theorem]{Proposition}

\newtheorem{conjecture}[theorem]{Conjecture}

\theoremstyle{definition}

\newtheorem{example}[theorem]{Example}

\makeatletter              
\let\c@equation\c@theorem  
\makeatother

\newcommand{\DD}{\mathbb{D}}
\newcommand{\FF}{\mathbb{F}}
\newcommand{\HH}{\mathbb{H}}

\newcommand{\be}{\begin{enumerate}}
\newcommand{\ee}{\end{enumerate}}
\newcommand{\bq}{\begin{eqnarray*}}
\newcommand{\eq}{\end{eqnarray*}}
\newcommand{\bqn}{\begin{eqnarray}}
\newcommand{\eqn}{\end{eqnarray}}

\begin{document}
\title[Rigidity of down-up algebras]
{Rigidity of down-up algebras \\
with respect to finite group coactions}

\author{J. Chen, E. Kirkman and J.J. Zhang}

\address{Chen: School of Mathematical Science,
Xiamen University, Xiamen 361005, Fujian, China}

\email{chenjianmin@xmu.edu.cn}

\address{Kirkman: Department of Mathematics,
P. O. Box 7388, Wake Forest University, Winston-Salem, NC 27109}

\email{kirkman@wfu.edu}

\address{Zhang: Department of Mathematics, Box 354350,
University of Washington, Seattle, Washington 98195, USA}

\email{zhang@math.washington.edu}

\begin{abstract}
If $G$ is a nontrivial finite group coacting on a
graded noetherian down-up algebra $A$ inner
faithfully and homogeneously, then the fixed subring
$A^{co\; G}$ is not isomorphic to $A$. Therefore
graded noetherian down-up algebras are rigid with
respect to finite group coactions, in the sense of
Alev-Polo. An example is given to show that this 
rigidity under group coactions does not have all 
the same consequences as the rigidity under group 
actions.
\end{abstract}







\maketitle


\setcounter{section}{-1}

\section{Introduction}
\label{xxsec0}

Throughout this paper, let $\Bbbk$ be a base field that
is algebraically closed of characteristic zero, and let all vector
spaces, (co)algebras, and morphisms be over $\Bbbk$.  

A remarkable theorem of Alev-Polo \cite[Theorem 1]{AP} states:

{\it Let ${\mathfrak g}$ and ${\mathfrak g}'$
be two semisimple Lie algebras. Let $G$ be a finite group of algebra
automorphisms of the universal enveloping algebra $U({\mathfrak g})$
such that the fixed subring $U({\mathfrak g})^G$ is isomorphic to 
$U({\mathfrak g}')$. Then $G$ is trivial and ${\mathfrak g}\cong 
{\mathfrak g}'$.}

Alev-Polo called this result a rigidity theorem for universal 
enveloping algebras. In addition, they proved a rigidity theorem for 
the Weyl algebras \cite[Theorem 2]{AP}. Kuzmanovich and the 
second- and third-named authors proved Alev-Polo's rigidity theorems 
in the graded case in \cite[Theorem 0.2 and Corollary 0.4]{KKZ1}.

 (Commutative) polynomial rings are not rigid; indeed,
by the classical Shephard-Todd-Chevalley Theorem if $G$ is a 
reflection group acting on a commutative polynomial ring $A$ 
then $A^G$ is isomorphic to $A$. Artin-Schelter
regular algebras \cite{AS} are considered to be a natural analogue of
polynomial rings in many respects. This paper concerns a
class of noncommutative Artin-Schelter regular algebras. The 
rigidity of a noncommutative algebra is closely related
to the lack of reflections in the noncommutative setting \cite{KKZ1}.
Therefore the rigidity of an algebra leads to a trivialization of the
Shephard-Todd-Chevalley theorem \cite{ST, KKZ2}, which is one of the key
results in noncommutative invariant theory \cite{Ki}.  The rigidity property is
also related to Watanabe's criterion for the Gorenstein property, see
\cite[Theorem 4.10]{KKZ3}. Some recent work in noncommutative algebraic
geometry connects the rigidity property and the lack of reflections to Auslander's
theorem \cite{BHZ}, which is one of the fundamental ingredients in the 
McKay correspondence \cite{CKWZ1, CKWZ2}. Further understanding of the rigidity 
property will have implications for several other research directions.

In \cite{KKZ5}, rigidity with respect to group coactions is
studied. Let $A$ be a connected ($\mathbb{N}$-)graded $\Bbbk$-algebra. 
A $G$-coaction on $A$ (preserving the ${\mathbb N}$-grading) is 
equivalent to a $G$-grading of $A$ (compatible with the original
${\mathbb N}$-grading), and the fixed subring $A^{co G}$ is $A_e$, 
the component of the unit element $e \in A$ under the $G$-grading. We 
recall a definition \cite[Definition 0.8]{KKZ5}:
 we say that a connected graded algebra $A$ is {\it rigid with
respect to group coactions} if for every nontrivial finite group $G$
coacting on $A$ homogeneously and inner faithfully, the fixed subring
$A^{co\; G}$ is not isomorphic to $A$ as algebras. The following
Artin-Schelter regular algebras are rigid with respect to group coactions
\cite[Theorem 0.9]{KKZ5}:
\begin{enumerate}
\item[(a)]
The homogenization of the universal enveloping algebra of a finite dimensional
semisimple Lie algebra.
\item[(b)]
The Rees ring of the Weyl algebra with respect to the standard
filtration.
\item[(c)]
The non-PI Sklyanin algebras of global dimension at least 3.
\end{enumerate}

These results can be viewed as dual versions of the rigidity theorems
proved in \cite[Theorem 0.2 and Corollary 0.4]{KKZ1}.

Down-up algebras were introduced by Benkart-Roby in \cite{BR}
as a tool to study the structure of certain posets. Graded
noetherian down-up algebras are Artin-Schelter regular algebras of
global dimension three with two generators \cite{KMP}. We 
recall the definition of only a graded noetherian down-up algebra.  For
$\alpha$ and $\beta$ scalars in $\Bbbk$, let the {\it down-up
algebra} $\DD(\alpha,\beta)$ be the algebra
generated by $u$ and $d$ and subject to the relations 
\begin{align}
\label{E0.0.1}\tag{E0.0.1}
u^2 d &=\alpha udu +\beta d u^2,\\
\label{E0.0.2}\tag{E0.0.2}
u d^2 &=\alpha dud +\beta d^2 u.
\end{align}
When $\alpha=0$ we denote the down-up algebra $\DD(0, \beta)$ by 
$\DD_{\beta}$. In this
paper we always assume that $\beta\neq 0$, or equivalently,
$\DD(\alpha,\beta)$ is a graded noetherian Artin-Schelter
regular algebra of global dimension three. The groups of
algebra automorphisms of down-up algebras (which depend upon the values 
of $\alpha$ and $\beta$) were computed in
\cite{KK}. These groups are rich enough to provide many nontrivial
examples. Some invariant theoretic aspects of down-up
algebras have been studied in \cite{KK, KKZ4}.

There is a rigidity result concerning group actions on down-up
algebras, see \cite[Proposition 6.4]{KKZ1}. The only theorem in this
paper is the following rigidity result for group coactions on graded
noetherian down-up algebras.

\begin{theorem}
\label{xxthm0.1} Let $A$ be a graded noetherian down-up algebra
$\DD(\alpha,\beta)$ and let $G$ be a nontrivial finite group
coacting on $A$ inner faithfully and homogeneously. Then the
fixed subring $A^{co\; G}$ is not Artin-Schelter regular.
As a consequence, $A$ is rigid with respect to finite group
coactions.
\end{theorem}

One can also consider general Hopf algebra actions on
the down-up algebras. The following conjecture is reasonable.

\begin{conjecture}
\label{xxcon0.2}
The graded noetherian down-up algebras are rigid with respect
to semisimple Hopf algebra actions in the sense of
\cite[Remark 0.10]{KKZ5}.
\end{conjecture}

In a slightly different language, Theorem \ref{xxthm0.1} says 
that graded noetherian down-up algebras do not admit 
``dual reflection groups'' for group coactions in the sense of
\cite[Definition 0.1]{KKZ5}. A result in \cite[Corollary 4.11]{KKZ3}
states: {\it Let $B$ be a noetherian Artin-Schelter 
regular domain, and let $G$ be a finite group acting on $B$
homogeneously. Suppose that $G$ contains no quasi-reflection. 
Then the fixed subring $B^G$ is Artin-Schelter Gorenstein 
if and only if the homological determinant of the $G$-action
is trivial.} When $A$ is a graded noetherian down-up algebra,
$A$ does not have a quasi-reflection of finite order by 
\cite[Proposition 6.4]{KKZ1}. Therefore $A^G$ is AS Gorenstein 
if and only if the $G$-action has trivial homological determinant. 
In \cite[Theorem 3.6]{KKZ3} it is shown that if $H$ is a 
semisimple Hopf algebra acting on an AS regular algebra with 
trivial homological determinant, then $A^H$ is AS Gorenstein; 
it is reasonable to ask whether the converse holds for algebras 
with no ``dual reflection groups".  Hence
Theorem \ref{xxthm0.1} suggests the following question in the 
group coaction setting: 
{\it Let $A$ be a graded noetherian down-up algebra with a finite group
$G$-coaction. If the fixed subring $A^{co\; G}$ is AS Gorenstein, must
the homological determinant of the $G$-coaction be trivial?} 

We conclude the paper with an example (Example \ref{xxex2.1}) that 
provides a negative answer to the above question,  indicating a 
difference between the invariant theory under group actions and the 
invariant theory under group coactions (or more generally, under 
Hopf algebra actions). It would be interesting to develop further 
tools that would determine precisely when,  under a $G$-coaction, 
the fixed subring $A^{co \; G}$ is AS Gorenstein, and, more generally, 
when the  homological determinant of a Hopf action being trivial is a 
necesary condition for $A^H$ to be AS Gorenstein.

\subsection*{Acknowledgments}
J. Chen was partially supported by the National Nature Science 
Foundation of China (Grant No. 11571286)
and the Natural Science Foundation of Fujian Province of China 
(Grant No. 2016J01031). E. Kirkman was partially supported by
grant \#208314 from the Simons Foundation. J.J. Zhang was partially 
supported by the US National Science Foundation (Grant No. DMS 1402863).

\section{Proof of Theorem \ref{xxthm0.1}}
\label{xxsec1}

Some basic definitions can be found in \cite{BB, KKZ5}; for example,
inner faithful is defined in \cite[Definition 2.7]{BB}.
Artin-Schelter regular will be abbreviated by AS regular; for the definition see
\cite[Definition 1.1]{KKZ5}.
We first recall some basic facts about down-up algebras
from \cite{BR, KK, KMP, KKZ4}.

\begin{lemma}
\label{xxlem1.1} Let $A$ be the down-up algebra
$\DD(\alpha,\beta)$ where $\beta\neq 0$.
\begin{enumerate}
\item[(1)]
$A$ is a connected graded noetherian AS regular domain of global
dimension three.
\item[(2)]
The Hilbert series of $A$ is $((1-t)^2(1-t^2))^{-1}$.
\item[(3)]
$\{ u^i(du)^j d^ k\mid i,j,k\geq 0\}$ is a $\Bbbk$-linear basis
of $A$.
\end{enumerate}
\end{lemma}

The following lemma is well-known.

\begin{lemma}
\label{xxlem1.2}
Let $A$ be a noetherian connected graded AS regular algebra of
global dimension three. Then $A$ is generated by either two or three
elements.
\end{lemma}

\begin{proof} This is \cite[Proposition 1.5]{AS} when $A$ is generated
in degree 1 and \cite[Proposition 1.1(i)]{Ste} when $A$ is not generated
in degree 1.
\end{proof}

It is well-known that, for every finite group $G$, a left
$(\Bbbk G)^\ast$-action on an algebra $A$ is equivalent to a right
$G$-coaction on $A$.  Since $\Bbbk$ is algebraically closed of
characteristic zero, if $G$ is a finite abelian group, the Hopf
algebra $\Bbbk G$ is isomorphic to its dual $(\Bbbk G)^\ast$. This
fact implies that a right $G$-coaction on $A$ is equivalent to some
left $G$-action on $A$. We will use these facts freely. The
following lemma is a rigidity result for abelian $G$-coactions.

\begin{lemma}
\label{xxlem1.3}
Let $A$ be a graded noetherian down-up algebra and $G$ be a
nontrivial finite abelian group coacting on $A$. Then the
fixed subring $A^{co\; G}$ is not AS regular.
\end{lemma}

\begin{proof} If $G$ is abelian, $\Bbbk G$ is isomorphic to
$(\Bbbk G)^{\ast}$ as Hopf algebras.
Since $G$ is abelian coacting on $A$, there is a $G$-action on $A$
such that $A^{co\; G}=A^G$. The assertion is a consequence
of \cite[Proposition 6.4]{KKZ1}.
\end{proof}

We consider, first, the case when $\alpha= 0$ and $G$ coacts 
homogeneously on the generators $u$ and $d$.  Note that 
although we show that $\DD_{\beta}$
is rigid with respect to these group coactions, the algebra 
can be graded by many different groups.

\begin{proposition}
\label{xxpro1.4} Let $A$ be the algebra $\DD_{\beta}$ and $G$ be a
nontrivial finite group coacting on $A$ such that $u$ and $d$ are
$G$-homogeneous. Then the fixed subring $A^{co\; G}$ is not AS regular.
\end{proposition}

\begin{proof} By Lemma \ref{xxlem1.3}, we can assume that $G$ is not abelian.
Let $\deg_G u=a$ and $\deg_G d=b$; then $G$ is generated by $\Re:=\{a,b\}$.
Using the relations $u^2 d=\beta du^2$ and $ud^2=\beta d^2 u$, we obtain
that $a^2$ and $b^2$ are in the center of $G$. This implies that
the orders of $a$ and $b$ are even. Let $i$, $j$, $k$ and $l$ be the smallest
positive integers such that $a^i=e$, $b^j=e$, $(ba)^k=e$ and $(ab)^l=e$.
Since $G$ is non-abelian, $i,j,k,l$ are all larger than 1, and both of $i$ and $j$
must be even. Then $x:=u^i$, $y:=d^j$, $z:=(du)^k$ and $t:=(ud)^l$ are
elements in $A^{co\; G}$.

Assume to the contrary that the fixed subring $A^{co\; G}$ is AS regular.
By Lemma \ref{xxlem1.2}, it is generated by at most three elements.
Choosing the generators of $A^{co\; G}$ carefully from lower degree to
higher, and using the fact that every monomial in $u$ and $d$ is
$G$-homogeneous, we can assume that $A^{co\; G}$ is generated by
$$h_s: =u^{m_s} (du)^{n_s} d^{p_s},$$
where $s$ is $1,2,3$ or $1,2$. If $m_s+n_s>0$ for all $s$, then
$y=d^j$ cannot be generated by $\{h_s\}_{s}$, a contradiction. Thus
$m_s+n_s=0$ for some $s$. Similarly, we have $n_s+p_s=0$ for some
$s$. These facts mean that we have
$$h_1=u^i=x \quad {\text{and}}\quad  h_2=d^j=y.$$
Since $x$ and $y$ do not generate $z$, $A^{co\; G}$ is generated
by three elements, namely, by $h_1=u^i$, $h_2=d^j$ and $h_3= u^m (du)^n d^p$.
If $m>0$, then $z=(du)^k$ cannot be generated by $h_1,h_2,h_3$.
Thus $m=0$. By symmetry, $p=0$. This implies that
$$h_3=(du)^n.$$
Since $i$ and $j$ are even, $t=(ud)^l$ cannot be generated by
$h_1,h_2,h_3$. Hence $A^{co\; G}$ is not generated by three (or fewer)
elements. This yields a contradiction by Lemma \ref{xxlem1.2}, and
therefore the fixed subring $A^{co\; G}$ is not AS regular.
\end{proof}

Next we consider an algebra  $\FF$ that is isomorphic to $\DD_{-1}$.
Let $\FF$ be the algebra generated by $x$ and $y$,
and subject to the two relations
\begin{equation}
\label{E1.4.1}\tag{E1.4.1}
x^3=yxy \quad \mbox{ and } \quad y^3=xyx.
\end{equation}

\begin{lemma}
\label{xxlem1.5} Retain the above notation.
\begin{enumerate}
\item[(1)]
$\FF$ is isomorphic to $\DD_{-1}$.
\item[(2)]
$\FF$ is a graded noetherian AS regular domain of global dimension three
with Hilbert series $((1-t)^2(1-t^2))^{-1}$.
\end{enumerate}
\end{lemma}

\begin{proof} (1) Setting $x=u+d$ and $y=u-d$, then the two relations
$x^3=yxy$ and $y^3=xyx$ are equivalent to the two relations
$u^2d=-du^2$ and $ud^2=-d^2u$. The assertion follows.

(2) This follows from the fact that all the assertions hold
for $\DD_{-1}$.
\end{proof}

Next we will apply Bergman's Diamond Lemma \cite{Be} to the algebra
$\FF$. By \cite{Be}, starting
with a set of initial relations, we can obtain a complete set 
of relations that is a
reduction system. Then every monomial (or word) becomes reduction-unique
by using this complete system.

\begin{lemma}
\label{xxlem1.6} Retain the notation as above.
\begin{enumerate}
\item[(1)]
Define an order on monomials by extending $x<y$ lexicographically. 
Then we have a complete set of five relations that is the reduction system
in the sense of \cite[p.180]{Be}.
$$\begin{aligned}
y^3& =xyx,\\
yxy&=x^3,\\
y^2x^3&=xyx^2y,\\
yx^2yx&=x^3 y^2,\\
yx^4&=x^4y.
\end{aligned}
$$
\item[(2)] We also have the other relations:
$$\begin{aligned}
y^4& =x^4,\\
yxyx&=x^4,\\
xyxy&=y^4.
\end{aligned}
$$
\item[(3)]
There is a $\Bbbk$-linear basis consisting of the monomials
of the form
$$x^i (yx^3)^j (yx^2)^{\epsilon} (y^2x^2)^k y^a x^b$$
where $i,j,k\geq 0$, $\epsilon$ is either $0$ or $1$,
and
$$(a,b)=(0,0), (1,0), (1,1), (1,2), (1,3), (2,0), (2,1), (2,2)$$
if $j+\epsilon+k=0$,
$$(a,b)=(1,0), (1,1), (1,2), (1,3), (2,0), (2,1), (2,2)$$
if $j>0$ and $\epsilon+k=0$, and
$$(a,b)=(1,0), (2,0), (2,1), (2,2)$$
if $\epsilon+k>0$.
\end{enumerate}
\end{lemma}

\begin{proof} (1) The assertion follows by direct computation. 
In fact, denote the relation $y^3 =xyx$ by (i), and the relation 
$yxy=x^3$ by (ii). Then, by using (i)+(ii), we have 
$y^3xy=y^2(yxy)=y^2x^3$ and $y^3xy=y^3(xy)=(xyx)(xy)=xyx^2y$. 
Thus we obtain the third relation (denote it by (iii)) in the list.
By using (ii)+(i), we have $yxy^3=(yx)y^3=(yx)xyx=yx^2yx$, 
and $yxy^3=(yxy)y^2=x^3y^2$, and thus the fourth relation 
(denote it by (iv)) in the list holds. Similarly, by using (ii)+(i), 
we obtain the fifth relation (denote it by (v)) in the list.  Then 
considering all the other possible
cases: (i)+(iii), (i)+(iv), (i)+(v),(ii)+(iii), (ii)+(iv) and (ii)+(v), 
there are no new relations.

(2) These assertions follow easily.

(3) Every monomial is of the form
$x^{i_1}y^{j_1}x^{i_2}y^{j_2}\cdots y^{j_{n-1}}x^{i_n}$,
where $n\geq 1$, $i_n\geq 0$, $j_s\geq 1$ and $i_s\geq 1$
for $ 1\leq s<n$. By the first relation $j_s$ can  be only 
1 or 2, and by the fifth relation in part (1), for $s>1$, 
$i_s$ can  be only 1, 2, or 3.

Take the last term $n$ into consideration. If $n=1$, we only have $x^{i_1}$. If $n=2$, we 
have only
$$x^i y, \quad x^i y^2, \quad x^i yx, \quad x^i yx^2,
\quad x^i yx^3, \quad x^i y^2 x, \quad x^i y^2 x^2$$
or
$$x^i y^{a} x^b$$
where $i\geq 0$ and 
\begin{equation}
\notag
(a,b)=(1,0), (1,1), (1,2), (1,3), (2,0), (2,1), (2,2).
\end{equation}
For $n\geq 3$, note that in the middle of the monomial,
$y^{j_{s-1}} x^{i_s}$ must be $yx^2$, or $yx^3$, or $y^2 x^2$.
Further, $yx^2 yx^3$, $y^2x^2 yx^3$, $y^2x^2 y x^2$, $y x^2 y x^2$
cannot appear in the middle of the monomial. Therefore we have
$$x^i (yx^3)^j (yx^2)^{\epsilon} (y^2x^2)^k y^a x^b$$
where $i,j,k\geq 0$, $\epsilon$ is either $0$ or $1$,
and
$$(a,b)=(0,0), (1,0), (1,1), (1,2), (1,3), (2,0), (2,1), (2,2)$$
if $j+\epsilon+k=0$,
$$(a,b)=(1,0), (1,1), (1,2), (1,3), (2,0), (2,1), (2,2)$$
if $j>0$ and $\epsilon+k=0$, and
$$(a,b)=(1,0), (2,0), (2,1), (2,2)$$
if $\epsilon+k>0$.
By Bergman's Diamond Lemma \cite{Be}, all monomials in part (3) form a
$\Bbbk$-linear basis of the algebra.
\end{proof}

Let $f$ be a monomial $x_1 x_2 x_3\cdots x_n$ in $\FF$ where $x_i$
is either $x$ or $y$. A {\it left subword} of $f$ is a monomial
of the form $x_1 x_2\cdots x_j$ for $j\leq n$, a {\it subword}
of $f$ is a monomial of the form $x_i\cdots x_j$ for some
$1\leq i\leq j\leq n$. Due to non-trivial relations in $\FF$,
if $ab=f$ for three monomials $a,b,f$, it is not necessarily true
that $a$ is a left subword of $f$.  The following lemma says that
in some special cases, $a$ must be a left subword of $f$.

\begin{lemma}
\label{xxlem1.7} Let $f$ be a subword of $(y^2 x^2)^s=yyxxyyxx\cdots yyxx$
for some $s\geq 1$. If $f=ab$ for some monomials $a,b$ in $\FF$, then
$a$ is a left subword of $f$.
\end{lemma}

\begin{proof} By changing $s$ to a larger number and adding more letters
to $f$ and $b$ from the right, we may assume that the degree of $f$ is at least
4.  We prove the assertion by induction on the degree of $a$. Nothing
needs to be proved if $a$ has degree 0. Now suppose $\deg a>0$. There are
four different cases for $f$:
$$f=yyxx\cdots, \quad yxxy\cdots, \quad xxyy\cdots, \quad  xyyx\cdots.$$
By Lemma \ref{xxlem1.6}(3), each $f$ is reduced. Suppose $f=yf'$ is in the first
two cases. If $a=ya'$, then, after canceling out $y$, we have $f'=a'b$, and the assertion
follows from the induction. If $a=xa'$, then the reduced form of $ab$ is
less than $f$ in the order used in the  Diamond Lemma, but this is impossible, and so we are done
in this case.  Suppose next that $f=xf'$ is in one of the last two cases. If $a=xa'$,
then, after canceling out $x$, we have $f'=a'b$, and the assertion
follows from the induction. The remaining case is $a=ya'$, and we need to consider the following two
separate cases for $f$.

If $f=xxyy\cdots$ and $a=ya'$, then $yyf$ is a subword of $(y^2x^2)^{s+1}$ and
$y^2ab=y^3 a'b=xyx a'b$, which is less than $yyf$ in the order used in the Diamond
Lemma, but this is impossible.
If $f=xyyx\cdots$ and $a=ya'$, then $yxf=yxxyyx\cdots$ and $yxya'b=
x^3 a'b$, which is less than $yxxyyx\cdots=yxf$ in the order used in the Diamond
Lemma, but this is impossible.

Combining the above assertions, the induction shows that $a$ is a left subword of 
the word $f$.
\end{proof}

\begin{proposition}
\label{xxpro1.8}
Let $A$ be the algebra $\FF$ and $G$ be a nontrivial finite group
coacting on $A$ such that $x$ and $y$ are $G$-homogeneous. Then
the fixed subring $A^{co\; G}$ is not AS regular.
\end{proposition}

\begin{proof} By Lemma \ref{xxlem1.3}, we need to
consider only the case when $G$ is not abelian. Assume to the contrary that
$\FF^{co\; G}$ is AS regular. By \cite[Lemma 3.3(2)]{KKZ5},
$\FF^{co\; G}$ has global dimension three.
Note that $\FF$ is a semigroup
algebra $\Bbbk T$ for the semigroup
$$T=\langle x,y\mid x^3=yxy, y^3=xyx\rangle$$
and $G$ is a finite factor group $T/N$ for some normal
subsemigroup $N$. We have
\begin{enumerate}
\item[(1)]
The fixed subring $\FF^{co\; G}$ is the semigroup ring $\Bbbk N$.
\item[(2)]
$\FF^{co\; G}$ is minimally generated by a finite subset
$S\subset N$.
\item[(3)]
Every monomial in $\FF^{co \; G}$ is a product of elements in $S$.
\end{enumerate}
Note that we have identified a monomial in $\FF$ with an element in
$T$. Let $g_1$ be the image of $x$ in $G$ and $g_2$ be the image
of $y$ in $G$. Then there is an $s>1$ such that
$(g_2g_2 g_1g_1)^s\in N$, or equivalently,  $(g_2g_2 g_1g_1)^s=e$
in $G$. Then $f_1=(yyxx)^s\in \FF^{co\; G}$. Similarly we have
three other monomials in $\FF^{co\; G}$:
$$\begin{aligned}
f_2&= yxx(yyxx)^{s-1} y,\\
f_3&= xx(yyxx)^{s-1} yy,\\
f_4&= x(yyxx)^{s-1} yyx.
\end{aligned}
$$
Then $f_1,f_2,f_3,f_4$ are four elements in $\FF^{co\; G}$.
Since $f_1$ is in $\FF^{co\; G}$, $f_1=a_1 b_1$ where $a_1$
is in the set $S$. Similarly, we have $a_i \in S$ such that
$f_i=a_i b_i$ for $i=2,3,4$. By Lemma \ref{xxlem1.7},
$a_i$ is a left subword of $f_i$. By Lemma \ref{xxlem1.6}(3),
as left subwords of $f_1,f_2,f_3,f_4$ respectively, $a_1,a_2,a_3,a_4$,
are reduced and linearly independent. Therefore the order of
$S$ is at least 4, which contradicts Lemma \ref{xxlem1.2}. Therefore
$\FF^{co\; G}$ is not AS regular.
\end{proof}

We consider another algebra $\HH$ that is generated by
$x$ and $y$ subject to the relations
$$\begin{aligned}
x^2 y+yx^2 -2y^3&=0,\\
-2x^3+xy^2+y^2 x&=0.
\end{aligned}
$$

The following lemma is similar to Lemma \ref{xxlem1.5}.

\begin{lemma}
\label{xxlem1.9} Retain the above notation.
\begin{enumerate}
\item[(1)]
$\HH$ is isomorphic to $\DD(-2,-1)$.
\item[(2)]
$\HH$ is a graded noetherian AS regular domain of global dimension three
with Hilbert series $((1-t)^2(1-t^2))^{-1}$.
\item[(3)]
For each $n>0$, elements $(xy)^n$ and $(yx)^n$ are linearly independent.
\item[(4)]
For each $n\geq 0$, elements $y(xy)^n$ and $x(yx)^n$ are linearly independent.
\end{enumerate}
\end{lemma}

\begin{proof} (1) Set $u=x-y$ and $d=x+y$, then the two relations
$u^2d=-2udu-du^2$ and $ud^2=-2dud-d^2u$ are equivalent to the two relations
$x^2 y+yx^2 -2y^3=0$ and
$-2x^3+xy^2+y^2 x=0$.

(2) Known for $\DD(\alpha, \beta)$ for all $\beta\neq 0$.

(3) For any $d \geq 2$, let $D_{2d}$ be the dihedral group of order $2d$ 
generated by $a$ and $b$ subject to the relations $a^2=b^2=(ab)^d=e$
where $e$ is the unit of $D_{2d}$.
Consider the $G:=D_{2d}$-coaction on $A$ obtained by setting
$\deg_G x=a$ and $\deg_G y=b$.  It is clear that $A$ is $G$-graded.  Choosing $d\gg n$, then
$$\deg_G((xy)^n)=(ab)^n\neq (ba)^n =\deg_G ((yx)^n),$$
which implies the assertion.

(4) The proof is similar to the proof of part (3), and so it is omitted.
\end{proof}

\begin{proposition}
\label{xxpro1.10}
Let $A$ be the algebra $\HH$ and $G$ be a nontrivial finite group
coacting on $A$ such that $x$ and $y$ are $G$-homogeneous. Then
the fixed subring $A^{co\; G}$ is not AS regular.
\end{proposition}

\begin{proof} By Lemma \ref{xxlem1.3}, we may assume that
$G$ is non-abelian. Suppose to the contrary that $A^{co\; G}$ is
AS regular. Let $g_1:= \deg_G x$ and $g_2:=\deg_G y$.  The 
$G$-grading forces $g_1^2=g_2^2$.  Let $a:=g_1^2=g_2^2$. If 
$a\neq e$, then both $x^2$ and $y^2$
are in the same $G$-graded component $A_{a}$. By \cite[Theorem 0.3(1)]{KKZ5},
$A_{a}$ is free of rank 1 over
$A^{co\; G}$. Thus $A_a=x^2 A^{co\; G}$ and
$A_a=y^2 A^{co\; G}$. This contradicts the fact that $x^2$ and $y^2$ are
linearly independent. Therefore $g_1^2=g_2^2=e$. As a consequence,
$G$ is isomorphic to the dihedral group $D_{2n}$ of order $2n$,
for some $n \geq 2$. When $n$ is odd the unique largest length element 
of $G$ with respect to the Coxeter generating set $\{g_1, g_2\}$ is
$m=(g_1 g_2)^{(n-1)/2} g_1=(g_2g_1)^{(n-1)/2} g_2$,
while when $n$ is even it
is $m=(g_1 g_2)^{n/2} =(g_2g_1)^{n/2}$, which is central.
By Lemma \ref{xxlem1.9}(3,4),
when $n$ is odd, $(xy)^{(n-1)/2} x$ and $(yx)^{(n-1)/2}y$
are linearly independent elements of degree $n$ in the $G$-graded 
component $A_m$, and when $n$ is even, $(xy)^{n/2}$ and $(yx)^{n/2}$
are linearly independent elements of degree $n$ in $A_m$. But the smallest
degree of elements in $A_m$ is $n$. This yields a
contradiction by \cite[Theorem 0.3(4)]{KKZ5}.
\end{proof}

For the rest of this section we assume that $G$ is a finite
group coacting on the down-up algebra $\DD(\alpha,\beta)$.

\begin{lemma}
\label{xxlem1.11} 
Let $A$ be the connected graded algebra $\DD(\alpha,\beta)$
with $G$-coaction. Let $x_1$ and $x_2$ be two linearly 
independent $G$-homogeneous elements in $A_1$. Suppose there 
are two nontrivial relations that hold in $A$:
$$f_1:=\sum_{i,j,k=1}^2 c_{i,j,k} x_i x_j x_k=0,$$
and
$$f_2:=\sum_{i,j,k=1}^2 e_{i,j,k} x_i x_j x_k=0$$
such that $c_{i,j,k}\neq 0$ and $e_{i,j,k}=0$
for some $(i,j,k)$. Then,  for
all $(i,j,k)$ with $e_{i,j,k}\neq 0$, the 
$x_ix_jx_k$ have the same $G$-degree.
\end{lemma}

\begin{proof}
Since the monomial $x_ix_jx_k$ does not appear in the relation 
$f_2$ with nonzero coefficient, $f_1$ and $f_2$ are linearly 
independent.  If $f_2$ contains two monomials with nonzero 
coefficients and different $G$-degrees, then $f_2$ must be a sum 
of  $G$-homogeneous pieces $g_1, \dots, g_n$ for $n \geq 2$, with 
each $G$-homogeneous piece a relation in $A$ of degree 3.  But 
then $f_1, g_1$, and $g_2$ are three linearly independent relations 
in $A$ of degree 3, which is a contradiction.
\end{proof}

\begin{proposition}
\label{xxpro1.12}
Suppose $G$ is a finite non-cyclic group coacting on
$A:=\DD(\alpha,\beta)$ homogeneously and inner faithfully.
Then one of the following occurs.
\begin{enumerate}
\item[(1)]
$\alpha=0$ and $u$ and $d$ are $G$-homogeneous after
a change of variables.
\item[(2)]
$A$ is isomorphic to $\FF$ and using the generators of $\FF$,
both $x$ and $y$ are $G$-homogeneous.
\item[(3)]
$A$ is isomorphic to $\HH$ and using the generators of $\HH$,
both $x$ and $y$ are $G$-homogeneous.
\item[(4)]
$G$ is abelian and there are linearly independent elements
$x$ and $y$ of $\DD(\alpha, -1)$ of degree one
such that
$$\begin{aligned}
\alpha x^2 y+(-2-\alpha) xyx+ \alpha yx^2 +(2-\alpha) y^3&=0,\\
(2-\alpha) x^3+ \alpha xy^2+(-2-\alpha) yxy+ \alpha y^2 x&=0
\end{aligned}
$$
and  $x$ and $y$ are $G$-homogeneous.
\item[(5)]
$G$ is abelian and $u$ and $d$ are $G$-homogeneous after a change
of variables.
\end{enumerate}
\end{proposition}

\begin{proof} Each of the five parts listed in
 (1-5) can occur. Part (5) could occur most often, so,
for the rest of the proof, we implicitly assume that we are
not in the situation of {\bf part (5)}.

Write $A_1=\Bbbk x+ \Bbbk y$ where $x,y$ are $G$-homogeneous.  
Then $g_1:=\deg_G x$ and $g_2:=\deg_G y$ generate $G$. Since 
$G$ is not cyclic, we have
$$e \neq g_1\neq g_2\neq e.$$

{\bf Case 1 $\alpha = 0, \beta=1$:} First, we assume that $A=\DD_1$.
Then, for any two linearly independent elements $x,y$ of $A$
of degree 1, one can check that $x^2$ and $y^2$ are central.
Therefore we can assume that $x$ and $y$ are $G$-homogeneous,
by the second paragraph. After changing $\{u,d\}$ to
$\{x,y\}$, we can assume that $u$ and $d$ are $G$-homogeneous.
Thus {\bf part (1)} holds for  $A=\DD_1$.

{\bf Case 2 $\alpha = 0, \beta \neq 1$:} Secondly, 
we assume that $A=\DD_{\beta}$ where $\beta\neq 1$.
As noted in the second paragraph, there are two elements $x$ and $y$
in degree 1 with different $G$-grades. We consider two cases.

{\bf Case 2a:}  $x=cu$ and $y=au+b d$ for some $a,b,c\in \Bbbk$. Since $x$ and
$y$ are linearly independent, $bc\neq 0$. If $a=0$, then we can choose
$x=u$ and $y=d$ after a change of variables, and the assertion in
{\bf part (1)} follows. Now we assume $abc\neq 0$. Up to another change
of variables, we have $x=u$ and $y=u+d$, or equivalently, $u=x$
and $d=y-x$. Then the two relations \eqref{E0.0.1} and \eqref{E0.0.2} become
\begin{align}
\label{E1.12.1}\tag{E1.12.1}
x^2 y-\beta y x^2+(\beta-1) x^3&=0,\\
\label{E1.12.2}\tag{E1.12.2}
xy^2 -\beta y^2 x+(1-\beta) x^3&=x^2 y+(1-\beta)xyx -\beta yx^2.
\end{align}
Combining these two relations, one obtains that
\begin{equation}
\label{E1.12.3}\tag{E1.12.3}
xy^2 -\beta y^2 x+(\beta-1)xyx=0.
\end{equation}
Then \eqref{E1.12.1} must be
$G$-homogeneous by Lemma \ref{xxlem1.11}. As a
consequence, $x$ and $y$ have the same $G$-grade,
which contradicts the fact that $G$ is not cyclic.

{\bf Case 2b:} Up to a change of variables, the remaining
case is when $u=x-y$ and $d=x-ay$ where $a\neq 0,1$.
The two relations \eqref{E0.0.1} and \eqref{E0.0.2} become
$$(1-\beta)x^3+(-a+\beta)x^2y+(-1+\beta)xyx
+(-1+a\beta)yx^2$$
$$+(a-\beta)xy^2
+(a-a\beta) yxy +(1-a\beta) y^2 x+(-a+a\beta)y^3=0$$
and
$$(1-\beta)x^3+(-a+\beta)x^2y+(-a+a\beta)xyx
+(-1+a\beta) yx^2$$
$$+(a^2-a\beta)xy^2+(a-a\beta)yxy+(a-a^2\beta)y^2x
+(-a^2+a^2\beta)y^3=0.$$
Since $a \neq 1$, by linear combination, we have
$$(1-\beta)x^3+(-a+\beta)x^2y+(-1+a\beta)yx^2+a(1-\beta) yxy=0$$
and
$$(1-\beta)xyx+(-a+\beta)xy^2+(-1+a\beta)y^2x
+a(1-\beta)y^3=0.$$
Note that $1-\beta\neq 0$. Since $G$ is not cyclic,
$-a+\beta=0$ and $-1+a\beta=0$ by Lemma \ref{xxlem1.11}.
Thus $\beta=a=-1$, and we have the relations $$x^3-yxy=0 \mbox{ and } y^3-xyx=0$$
which is {\bf part (2)}.

{\bf Case 3 $\alpha \beta \neq 0$:} Thirdly, we assume that $A=\DD(\alpha,\beta)$ where
$\alpha \beta\neq 0$. As before we  need
to consider two cases by Lemma \ref{xxlem1.11}.

{\bf Case 3a}: Let $x=u$ and $y=u+d$, or equivalently
$u=x$ and $d=y-x$. The relations \eqref{E0.0.1}
and \eqref{E0.0.2} become
$$x^2y -\alpha xyx -\beta yx^2+(-1+\alpha+\beta) x^3=0$$
and
$$(1-\alpha-\beta)x^3+(-1+\alpha)x^2y+(-1+\beta)xyx$$
$$+(\alpha+\beta) yx^2+ xy^2+(-\alpha) yxy+(-\beta)y^2 x=0.$$
By adding these two relations, we obtain that
$$\alpha x^2y +(-1-\alpha+\beta) xyx+\alpha yx^2
+ xy^2-\alpha yxy-\beta y^2 x=0.$$
If $-1+\alpha+\beta\neq 0$, this forces  $\deg_G(x^2y)=\deg_G(xy^2)$.
This implies that $\deg_G(x)=\deg_G(y)$ and $G$ is cyclic,
a contradiction. If $-1+\alpha+\beta=0$, we still have
$\deg_G(x^2y)=\deg_G(xyx)$, which implies that $G$ is abelian.
Further, by Lemma \ref{xxlem1.11}, up to a common scalar,
the relation 
$$x^2y -\alpha xyx -\beta yx^2=0$$ 
must coincide with the relation 
$$\alpha x^2y +(-1-\alpha+\beta) xyx+\alpha yx^2=0.$$
As a consequence, $(\alpha, \beta)=(2,-1)$, and we have (a special case of) 
{\bf part (4)}.

{\bf Case 3b:} The remaining case is when $u=x-y$ and $d=x-ay$ where
$a\neq 0,1$. Using these generators to expand the relations \eqref{E0.0.1}
and \eqref{E0.0.2} we have
$$\begin{aligned}
(1-\alpha-\beta)x^3&+(-a+\alpha+\beta) x^2y
+(-1+a\alpha+\beta) xyx+(-1+\alpha+a\beta) yx^2\\
+(a-a\alpha-\beta) &xy^2+(a-\alpha-a\beta) yxy
+(1-a\alpha-a\beta)y^2 x+a(-1+\alpha+\beta) y^3=0\\
\end{aligned}$$
{and}\\
$$\begin{aligned}
(1-\alpha-\beta)x^3&+(-a+a\alpha+\beta)x^2y
+(-a+\alpha+a\beta) xyx+(-1+a\alpha+a\beta)yx^2+\\
a(a-\alpha-\beta) &xy^2+a(1-a\alpha-\beta) yxy
+a(1-\alpha-a\beta)y^2x+a^2(-1+\alpha+\beta)y^3
=0.
\end{aligned}
$$
If $1-\alpha-\beta=0$, then the  two relations above
become
$$\begin{aligned}
(1-a) x^2y&
+\alpha(a-1) xyx+\beta(a-1) yx^2\\
&+\beta(a-1) xy^2+\alpha(a-1) yxy
+(1-a)y^2 x=0\\
\end{aligned}$$
and
$$\begin{aligned}
\beta(1-a)x^2y&
+\alpha(1-a) xyx+(a-1)yx^2\\
&+a(a-1) xy^2+a\alpha(1-a) yxy
+a\beta(1-a)y^2x   =0.  
\end{aligned}
$$
Since $a\neq 1$, after dividing by $1-a$, we have
$$x^2y-\alpha xyx-\beta yx^2
-\beta xy^2-\alpha yxy
+y^2 x=0$$
and
$$\beta x^2y +\alpha xyx-yx^2
-a xy^2+a\alpha yxy
+a\beta y^2x=0
$$
which are equivalent to
$$\begin{aligned}
(1+\beta)x^2y+(-1-\beta)yx^2
+(-\beta-a)xy^2+\alpha (a-1) yxy
+(1+a\beta)y^2 x&=0\\
\end{aligned}$$
and
$$\begin{aligned}
(a+\beta) x^2y
+\alpha(1-a)xyx+(-a\beta-1)yx^2
+(-a\beta -a) xy^2+(a+a\beta) y^2x&=0.
\end{aligned}
$$
If $\beta\neq -1$, by Lemma \ref{xxlem1.11}, we have
$\deg_G(xyx)=\deg_G(y^2x)$, which implies that
$G$ is cyclic, a contradiction. If $\beta=-1$
(and then $\alpha=2$),
then the above two relations become
$$\begin{aligned}
xy^2-2 yxy +y^2 x&=0,\\
x^2y-2 xyx+ yx^2&=0,
\end{aligned}
$$
so by Lemma \ref{xxlem1.11} $G$ is abelian. Therefore
we have {\bf part (4)}.

If $1-\alpha-\beta\neq 0$, the two relations given at the
beginning of Case 3b are equivalent to
$$\begin{aligned}
\alpha(1-a)x^2y+&(a-1)(1+\alpha-\beta)xyx+\alpha (1-a) yx^2+(1-a)(a-\beta) xy^2\\
+\alpha(a^2-1) & yxy+(1-a)(1-a\beta) y^2x
+a(1-a)(-1+\alpha+\beta) y^3=0,\\
\end{aligned}$$
and
$$\begin{aligned}
(1-a)(1-\alpha-\beta)&x^3+(a-1)(a-\beta)x^2 y +\alpha(1-a^2) xyx+
(1-a)(-1+a\beta) yx^2\\
+a\alpha(a-1) &xy^2+ a(1-a) (1+\alpha-\beta) yxy+ a\alpha (a-1) y^2x=0 .
\end{aligned}
$$
Since $a\neq 1$, we can simplify them to obtain the following two relations
$$\begin{aligned}
\alpha x^2y-(1+\alpha-\beta)xyx+&\alpha yx^2+(a-\beta) xy^2-\alpha(a+1)  yxy\\
 &+(1-a\beta) y^2x+a(-1+\alpha+\beta)y^3=0,
\end{aligned}
$$
and
$$\begin{aligned}
 (1-\alpha-\beta)x^3-(a-\beta)x^2 y +&\alpha(1+a) xyx+(-1+a\beta) yx^2\\
&-a\alpha xy^2
+ a (1+\alpha-\beta) yxy- a\alpha  y^2x=0 .
\end{aligned}
$$

Suppose $a\neq -1$. Since $\alpha (a-1)\neq 0$, the coefficients
of $yxy$ and $y^3$ are nonzero. By Lemma \ref{xxlem1.11},
we have $\deg_G(yxy)=\deg_G(y^3)$. This forces $G$ to be cyclic,
a contradiction. If $a=-1$ and $\beta\neq -1$, a similar argument
leads to a contradiction. If $a=-1=\beta$, we obtain two
relations
$$\begin{aligned}
\alpha x^2 y+(-2-\alpha) xyx+ \alpha yx^2 +(2-\alpha) y^3&=0,\\
(2-\alpha) x^3+ \alpha xy^2+(-2-\alpha) yxy+ \alpha y^2 x&=0.
\end{aligned}
$$
In addition, if $\alpha\neq -2$, then the coefficients of $x^2y$ and $xyx$ are nonzero. 
By Lemma
\ref{xxlem1.11}, $\deg(x^2y)=\deg(xyx)$ which implies that $G$ is abelian, and so we have {\bf part (4)};
 if $\alpha=-2$, the two relations become
$$\begin{aligned}
 x^2 y+ yx^2 -2y^3&=0,\\
-2 x^3+ xy^2+ y^2 x&=0,
\end{aligned}
$$
which is {\bf part (3)}. This completes the proof.
\end{proof}

\begin{proof}[Proof of Theorem \ref{xxthm0.1}]
If $G$ is abelian, then the assertion follows from
Lemma \ref{xxlem1.3}. We now assume that $G$ is non-abelian.
By Proposition \ref{xxpro1.12} there are three cases
(parts (1), (2) and (3) occur). The first case follows from
Proposition \ref{xxpro1.4}, the second case follows
from Proposition \ref{xxpro1.8}, and the third case follows
from Proposition \ref{xxpro1.10}.
\end{proof}

\section{An Example}
\label{xxsec2}

In this short section we provide an example that negatively 
answers the question  mentioned at the end of the introduction.

\begin{example}
\label{xxex2.1}
Consider the algebra $\DD:=\DD(0,1)$, and
let $G$ be the dihedral group 
$$D_{8}=\{e, \rho, \rho^2, \rho^3, r, \rho r, \rho^2 r, \rho^3 r\},$$
where $\rho$ is a rotation of order $4$ and $r$ is a reflection.
Let $g_1=r$ and $g_2=\rho$. Then $\{g_1,g_2\}$ generate the group
$G$. Define the $G$-coaction on $\DD$ by setting $\deg_G u=\rho$
and $\deg_G d=r$. By a computation similar to \cite[Example 7.1]{KKZ3},
one sees that the homological co-determinant of the $G$-coaction is 
$g_1^2 g_2^2=\rho^2$, which is not trivial. In other words, 
the $(\Bbbk G)^{\circ}$-action on $\DD$ has nontrivial 
homological determinant. On the other hand, by Lemma \ref{xxlem2.2}
below, the fixed subring $\DD^{co \; G}$ is 
$\Bbbk [(du)^2, (ud)^2,d^4, u^2]$ which is isomorphic to 
$\Bbbk [x,y,z,t]/(xy-zt^2)$.  
As a consequence, $\DD^{co\;  G}$ is a commutative complete 
intersection (and hence AS Gorenstein), but not AS regular. 
\end{example}

\begin{lemma}
\label{xxlem2.2}
Retain the notation as the above example. Let $\DD_{g}$ be the
$g$-component of $\DD$ for all $g\in D_{8}$.
\begin{enumerate}
\item[(1)]
$\DD_{e}=\DD^{co\; G}=\Bbbk [(du)^2, (ud)^2,d^4, u^2]$.
\item[(2)]
$\DD_{\rho}=u \DD^{co\;  G} \cong \DD^{co\;  G} (-1)$. 
\item[(3)]
$\DD_{\rho^2}=u^2 \DD^{co\;  G} \cong \DD^{co\;  G} (-2)$. 
\item[(4)] 
$\DD_{\rho^3}=u^3 \DD^{co\;  G}+dud \DD^{co\;  G}$.
\item[(5)]
$\DD_{r}=d \DD^{co\;  G}+udu \DD^{co\;  G}$.
\item[(6)]
$\DD_{\rho r}=ud \DD^{co\;  G}+u^2 du \DD^{co\;  G}\cong u \DD_{r}$.
\item[(7)]
$\DD_{\rho^2 r}=u^2 d \DD^{co\;  G}+u^3 du \DD^{co\;  G} \cong u^2 \DD_{r}$.
\item[(8)]
$\DD_{\rho^3 r}=u^3 d\DD^{co\;  G}+du \DD^{co\;  G}$.
\end{enumerate}
\end{lemma}

\begin{proof}
By Lemma \ref{xxlem1.1}(3),
every $G$-homogeneous element in $\DD$ is a linear combination 
of monomials of the form $u^i (du)^j d^k$ for some $i,j,k\geq 0$.  
Note that 
$$\deg_G d^k=\begin{cases} r & {\text{ $k$ is odd}}\\
1& {\text{ $k$ is even}}\end{cases}, \quad {\text{and}}\quad
\deg_G (du)^j=\begin{cases} r\rho & {\text{ $j$ is odd}}\\
1& {\text{ $j$ is even}}\end{cases}.$$
This implies that
$$\deg_G (du)^j d^k=\begin{cases} \rho^3 & {\text{ $j$ and $k$ are odd}}\\
r& {\text{ $j$ is even and $k$ is odd}}\\
r\rho &{\text{ $j$ is odd and $k$ is even}}\\
1 & {\text{ $j$ and $k$ are  even.}}
\end{cases}
$$
Hence
\begin{equation}
\label{E2.2.1}\tag{E2.2.1}
\deg_G u^i (du)^j d^k=\begin{cases} \rho^{i+3} & {\text{ $j$ and $k$ are odd}}\\
\rho^i r& {\text{ $j$ is even and $k$ is odd}}\\
\rho^i  r\rho=\rho^{i-1} r &{\text{ $j$ is odd and $k$ is even}}\\
\rho^i & {\text{ $j$ and $k$ are  even.}}
\end{cases}
\end{equation}

(1) It is clear that elements $(du)^2, (ud)^2,d^4, u^2$ are in
the fixed subring. If $u^i (du)^j d^k$ is in $\DD^{co\; G}$, then 
formula \eqref{E2.2.1} shows that
this could happen only when $j$ and $k$ are odd and $i\equiv 1 \mod 4$, 
or when $j$ and $k$ are  even and $i\equiv 0 \mod 4$. The assertion
follows.

(2) If $u^i (du)^j d^k$ is in $\DD_{\rho}$, then formula \eqref{E2.2.1} 
shows that
this could happen only when $j$ and $k$ are odd and $i+3\equiv 1 \mod 4$, 
or when $j$ and $k$ are  even and $i\equiv 1 \mod 4$. In both cases, $i\geq 1$.
Thus $u^i (du)^j d^k\in u\DD$, and, as a consequence, $u^i (du)^j d^k\in u\DD^{co\;  G}$.

(3) The proof is similar to the proof of (1,2).

(4) Continue the computation in the proof of parts (1,2), if 
$\deg_G u^i (du)^j d^k=\rho^3$, then this could  happen only when 
$j$ and $k$ are odd and $i+3\equiv 3 \mod 4$, 
or when $j$ and $k$ are  even and $i\equiv 3 \mod 4$. In the first case,
$i$ could be 0 and $u^i (du)^j d^k\in dud \DD$; and in the second case,
$i$ could be 3 and $u^i (du)^j d^k\in u^3 \DD$. This implies that
$u^i (du)^j d^k\in u^3 \DD^{co\;  G}+dud \DD^{co\;  G}$.

(5-8) The proofs are similar to the proof of (4) and omitted.
\end{proof}

\end{document}